\documentclass[reqno]{amsart}
\usepackage{amsmath, amsthm, amscd, amssymb, amsfonts, amsbsy}
\usepackage{latexsym, color, enumerate}
\usepackage{mathrsfs}
\usepackage{pxfonts}
\usepackage{bbm}
\usepackage{mathtools}
\usepackage{graphicx}
\usepackage{tikz}
\usetikzlibrary{arrows.meta,decorations.pathreplacing}
\usepackage{enumerate}
\usepackage{todonotes}
\usepackage{marginnote}

\theoremstyle{plain}
\newtheorem{theorem}[equation]{Theorem}
\newtheorem{lemma}[equation]{Lemma}
\newtheorem{corollary}[equation]{Corollary}

\theoremstyle{definition}

\theoremstyle{remark}
\newtheorem{remark}[equation]{Remark}

\newcommand{\dist}{\operatorname{dist}}
\newcommand{\diam}{\operatorname{diam}}

\numberwithin{equation}{section}

\providecommand{\abs}[1]{\lvert#1\rvert}

\providecommand{\norm}[1]{\lVert#1\rVert}

\renewcommand{\vec}[1]{\boldsymbol{#1}}

\begin{document}
\title[Boundary blow-up solutions]
{Boundary blow-up solutions: gradient asymptotics and uniqueness}

\author[S. Kim]{Seick Kim}
\address[S. Kim]{Department of Mathematics, Yonsei University, 50 Yonsei-Ro, Seodaemun-gu, Seoul 03722, Republic of Korea}
\email{kimseick@yonsei.ac.kr}
\thanks{S. Kim is supported by the National Research Foundation of Korea (NRF) under agreement NRF-2022R1A2C1003322.
S. Kim is affiliated with the Intelligent Computational Science Institute (IN2CSI), Yonsei University, 50 Yonsei-Ro, Seodaemun-gu, Seoul 03722, Republic of Korea.}


\begin{abstract}
Let $\Omega\subset\mathbb R^n$ be a bounded domain, and let $f$ be a nonnegative, nondecreasing function satisfying the Keller--Osserman condition. We study boundary blow-up solutions of $\Delta u=f(u)$ in $\Omega$. Although existence is classical, uniqueness under these assumptions is known in balls but remains open even for smooth convex domains. We identify the normalized gradient $Q_u=|\nabla u|^2/(2F(u))$, $F'=f$, as a quantity governing uniqueness.
Under a structural condition on $f$, a boundary blow-up solution $u$ is unique if $\limsup_{x\to\partial\Omega}Q_u(x)\le 1$, without any regularity assumption on $\partial \Omega$. 
For $C^{1,1}$ domains, assuming a growth condition on $f$, we prove $Q_u(x)\to 1$ for every boundary blow-up solution and hence obtain uniqueness under the structural condition. For convex domains, we prove  $Q_u\le 1$ for the minimal boundary blow-up solution and obtain uniqueness when $\sqrt F$ is eventually convex, without imposing any additional boundary regularity.
\end{abstract}
\maketitle

\section{Introduction}		
Let $f:\mathbb R\to[0,\infty)$ be locally Lipschitz continuous and nondecreasing, and suppose that it satisfies the Keller--Osserman condition
\begin{equation}\label{eq1.1}
\int^\infty\frac{dt}{\sqrt{F(t)}}<\infty,
\end{equation}
where $F$ is a primitive of $f$ that is positive for all sufficiently large $t$. Here and throughout, we suppress the lower limit when it is immaterial. The Keller--Osserman condition implies $F(t)\to\infty$. Since any two primitives of $f$ differ by a constant, they are eventually positive and comparable; hence the condition is independent of the choice of $F$.

These are our standing assumptions on $f$. Typical examples are $f(t)=t_+^p$ with $p>1$, where $t_+:=\max\{t,0\}$, and $f(t)=e^{ct}$ with $c>0$.

Let $\Omega\subset\mathbb{R}^n$ be a bounded domain.
We study the \emph{boundary blow-up problem}
\begin{equation}\label{eq1.2}
\begin{aligned}
\Delta u &= f(u) && \text{in }\;\Omega,\\
u(x) &\to \infty && \text{as }\;d(x):=\operatorname{dist}(x,\partial\Omega)\to0.
\end{aligned}
\end{equation}
A solution of \eqref{eq1.2} is called a \emph{boundary blow-up solution} or a \emph{large solution}.

Keller \cite{K} and Osserman \cite{O} established the basic existence theory and identified \eqref{eq1.1} as the natural growth condition for boundary blow-up.
Uniqueness, however, has proved substantially more delicate.
Various uniqueness results are known under additional assumptions; see, among others, Bandle and Marcus \cite{BM1,BM2}, Lazer and McKenna \cite{LM}, Marcus and V\'eron \cite{MV1,MV2,MV3}, and Dong, Kim, and Safonov \cite{DKS2008}.
For general nonlinearities under only the standing assumptions, uniqueness is known in balls \cite{CD2010, RW}, where radial symmetry reduces the problem to an ODE, but, to the best of our knowledge, remains open beyond balls, even in smooth convex domains such as ellipsoids.

A uniqueness result particularly relevant here is due to Costin, Dupaigne, and Goubet \cite{CDG2012}, who proved uniqueness in $C^3$ mean-convex domains when $\sqrt{F}$ is convex on $[t_0,\infty)$ for some $t_0\in\mathbb R$.
Here mean-convex means that the boundary has nonnegative mean curvature.
A key ingredient is the gradient bound
\[
\frac{\abs{\nabla u}^2}{2F(u)}\le1,
\]
obtained from the boundary geometry through a maximum principle argument.

The purpose of this paper is to isolate a uniqueness mechanism related to \cite{CDG2012} but independent of the geometry behind the gradient estimate.
For a boundary blow-up solution $u$, define
\[
Q_u(x):=\frac{|\nabla u(x)|^2}{2F(u(x))}.
\]
Our first main result is a uniqueness theorem on an arbitrary bounded domain:
if a boundary blow-up solution satisfies
\begin{equation}\label{eq1.3}
\limsup_{x\to\partial\Omega}Q_u(x)\le1,
\end{equation}
then, under the structural condition below, it is unique.
No regularity of $\partial\Omega$ is required.
This formulation separates uniqueness from the derivation of the gradient bound.
Once \eqref{eq1.3} is established by geometric or analytic means, the same comparison argument applies without further boundary information.

The proof uses a one-parameter family of transforms.
For $0<\alpha\le1$, define
\[
h_\alpha(t):=\int_t^\infty F(s)^{-\alpha/2}\,ds.
\]
Whenever this integral is finite,
\[
\Delta h_\alpha(u) = F(u)^{-\alpha/2}f(u)(\alpha Q_u-1).
\]
If $0<\alpha<1$, \eqref{eq1.3} gives $Q_u<1/\alpha$ near $\partial\Omega$.
Theorem~\ref{thm3.2} assumes that $F^{-\alpha/2}f$ is nondecreasing; together with $1-\alpha Q_u \ge0$, this yields uniqueness.

Theorem~\ref{thm3.9} shows that eventual monotonicity suffices: uniqueness holds if, for some  $0<\alpha<1$ and $t_0\in\mathbb R$,
\begin{equation}\label{eq1.5}
\int_{t_0}^\infty F(t)^{-\alpha/2}\,dt<\infty,
\qquad
t\mapsto F(t)^{-\alpha/2}f(t)\quad\text{is nondecreasing on }\;[t_0,\infty).
\end{equation}

The conditions in \eqref{eq1.5} are not nested in $\alpha$.
If $0<\alpha_1<\alpha_2\le1$, then
\[
\int^\infty F^{-\alpha_1/2}<\infty \implies \int^\infty F^{-\alpha_2/2}<\infty,
\]
whereas
\[
F^{-\alpha_1/2}f =(F^{-\alpha_2/2}f) \,F^{(\alpha_2-\alpha_1)/2}
\]
shows that monotonicity for $\alpha_2$ implies monotonicity for $\alpha_1$.
Thus, as $\alpha$ increases, the integrability condition in \eqref{eq1.5} weakens while the monotonicity condition strengthens.
For $\alpha=1$, the integrability condition is precisely the Keller--Osserman condition.

We next give a setting in which \eqref{eq1.3} follows from the boundary asymptotics of the gradient.
Suppose that $\Omega$ is a bounded $C^{1,1}$ domain and
\begin{equation}\label{eq1.6}
\limsup_{t\to\infty}\psi(t) \, \frac{f(t)}{\sqrt{F(t)}}<\infty,	\qquad
\psi(t):=\int_t^\infty\frac{ds}{\sqrt{2F(s)}}.
\end{equation}
For any boundary blow-up solution $u$, set $v:=\psi(u)$.
Theorem~\ref{thm4.8} gives
\[
v \in C^1(\overline\Omega), \qquad
v=0,\quad \nabla v=\vec n \;\text{ on }\;\partial\Omega,
\]
where $\vec n$ denotes the inward unit normal. Hence
\[
Q_u(x)=\abs{\nabla v(x)}^2 \to1\qquad\text{as }x\to\partial\Omega.
\]
If \eqref{eq1.5} holds for some $0<\alpha<1$, Theorem~\ref{thm3.9} yields uniqueness.

At the endpoint  $\alpha=1$, the boundary limit does not determine the sign of $1-Q_u$.
For bounded $C^3$ mean-convex domains, a maximum principle argument gives $Q_u\le1$;
under the eventual convexity of $\sqrt F$, this yields uniqueness and recovers \cite{CDG2012}.
For arbitrary bounded convex domains, the same estimate follows by approximation and again yields uniqueness, without any boundary regularity.

Thus the uniqueness argument separates into two parts: control of $Q_u$ and a comparison principle for the transformed solutions.
It applies whenever the required control of $Q_u$ is available.

We next briefly compare our assumptions with existing uniqueness criteria, leaving the details to the Appendix.

One approach, originating with Marcus and V\'eron \cite{MV3} and developed by L\'opez-G\'omez and Maire \cite{LGM2017}, assumes superadditivity up to a constant:
\[
f(a+b)\ge f(a)+f(b)-C,\qquad a,b\ge 0,
\]
for some $C \ge 0$.
Under this assumption, uniqueness holds on $C^1$ domains satisfying a uniform interior sphere condition.
Lemma~\ref{lem6.10} shows that, for every $0<\alpha<1$, \eqref{eq1.5} need not imply superadditivity up to a constant, even for smooth strictly increasing $f$.
Hence our criterion is not subsumed by \cite{LGM2017}.
The examples in Lemma~\ref{lem6.10} also fall outside the criteria of \cite{ADR2015} and \cite{GM2009}.
The former assumes that $f(t)/t$ is nondecreasing for all sufficiently large $t$, while the latter assumes the same of $f(t)/t^p$ for some $p>1$, which in particular implies eventual monotonicity of $f(t)/t$.
Either assumption implies superadditivity up to a constant.

At the endpoint $\alpha=1$, superadditivity up to a constant does follow; see Lemma~\ref{prop2.17}.
Hence \cite{LGM2017} applies under its $C^1$ boundary and uniform interior sphere assumptions. Our result is different in that it applies to arbitrary bounded convex domains, including domains with nonsmooth boundary.

A different comparison concerns the increment function.
L\'opez-G\'omez, Maire, and V\'eron \cite{LGMV2020} instead considered the increment function
\[
g(\ell):=\inf_{s\ge0}\,\{f(s+\ell)-f(s)\},\qquad \ell\ge 0.
\]
They proved uniqueness on Lipschitz domains under suitable assumptions on $g$.
Lemma~\ref{prop2.10} shows that \eqref{eq1.5} controls only its tail behavior.
Indeed, \eqref{eq1.5} allows $f$ to be constant on a bounded interval of length $\ell$, in which case $g(\ell)=0$. 
Thus \eqref{eq1.5} implies the Keller--Osserman condition for $g$, but not the positivity assumption imposed in \cite{LGMV2020}.

Lemma~\ref{lem2.24} shows that \eqref{eq1.6} implies the profile condition
\[
\liminf_{t\to\infty}\frac{\psi(\beta t)}{\psi(t)}>1
\qquad (0<\beta<1),
\]
appearing in the work of Bandle and Ess\'en \cite{BE}; see also \cite{BM2, Ma}.
In this line of work, such profile conditions are used to describe the boundary asymptotics of $u$ in terms of the one-dimensional blow-up profile $\psi^{-1}$.
By contrast, in our analysis \eqref{eq1.6} is used to obtain the boundary regularity $\psi(u)\in C^1(\overline\Omega)$; see Theorem~\ref{thm4.8}.
The stronger assumption \eqref{eq1.6} also excludes the borderline family $f(t)=t(\log t)^p$, $p>2$, treated by Alarc\'on, D\'iaz, and Rey \cite{ADR2015}.
For other related results on uniqueness and boundary asymptotics, see, for example, \cite{dPL2002, DG2004, Kim, LG2006, PV2006}.
The scope of our criterion is illustrated by Lemma~\ref{lem6.10}, whose examples satisfy both \eqref{eq1.5} and \eqref{eq1.6}, and hence give nonlinearities $f$ for which the boundary blow-up problem has a unique solution in a $C^{1,1}$ domain, while not being covered by the uniqueness results discussed above.

Finally, Dong, Kim, and Safonov \cite{DKS2008} established uniqueness for more general elliptic operators, in some cases on substantially less regular domains than those considered here.
Their main results, however, are restricted to the model nonlinearities $f(t)=t_+^p$ for $p>1$ and  $f(t)=e^{ct}$ for $c>0$.
Both satisfy \eqref{eq1.5} for some $\alpha \in (0,1)$ and \eqref{eq1.6}.

The paper is organized as follows.
Section~\ref{sec2} collects preliminary results on existence and comparison for boundary blow-up solutions.
Section~\ref{sec3} proves uniqueness on arbitrary bounded domains under assumption \eqref{eq1.3}.
Section~\ref{sec4} establishes boundary gradient asymptotics in bounded $C^{1,1}$ domains, and derives the corresponding uniqueness result.
Section~\ref{sec5} treats uniqueness in convex domains, corresponding to the endpoint case $\alpha=1$.
The main results are Theorems~\ref{thm3.9}, \ref{thm4.8}, and \ref{thm5.3}.

\section{Preliminaries}\label{sec2}

In this section, we briefly discuss the existence results of Keller \cite{K},
Osserman \cite{O}, and Loewner and Nirenberg \cite{LN}. We also introduce
some terminology which will be used in the later parts of the paper.
We begin with a simple lemma.

\begin{lemma}[Comparison principle]		\label{lem2.1}
Let $\Omega\subset\mathbb R^n$ be a bounded domain.
Let $u,v\in C^2(\Omega)$ satisfy
\[
\Delta u\ge f(u), \qquad \Delta v\le f(v) \quad\text{in }\Omega.
\]
If $\liminf_{x\to\partial\Omega}\,(v-u)(x)\ge0$, then $v\ge u$ in $\Omega$.
\end{lemma}
\begin{proof}
Suppose, to the contrary, that $u>v$ at some point in $\Omega$. For sufficiently small $\epsilon>0$,
\[
D_\epsilon:=\{u-v>\epsilon\}\ne\emptyset, \qquad \overline D_\epsilon \subset\Omega.
\]
Set $w:=u-v-\epsilon$.
Then $w=0$ on $\partial D_\epsilon$, and, since $u>v$ in $D_\epsilon$ and $f$ is nondecreasing,
\[
\Delta w \ge f(u)-f(v) \ge0 \qquad\text{in }\;D_\epsilon.
\]
The maximum principle gives $w\le0$ in $D_\epsilon$, a contradiction.
\end{proof}

\begin{remark}	\label{rmk2.2}
Let $\Omega_1,\Omega_2\subset\mathbb R^n$ be bounded domains with $\Omega_1\Subset\Omega_2$, and let $u_i$ be a boundary blow-up solution in $\Omega_i$, $i=1,2$. Then Lemma~\ref{lem2.1} gives
\[
u_1\ge u_2\qquad\text{in }\;\Omega_1.
\]
\end{remark}

The next theorem is due to Keller \cite{K}; see also Osserman \cite{O}.
\begin{theorem}	\label{thm2.5}
Let $u \in C^2(\Omega)$ satisfy $\Delta u=f(u)$ in a bounded domain $\Omega$.
There exists a continuous decreasing function $g:(0,\infty)\to\mathbb R$, depending only on $f$, such that
\[
\lim_{t\downarrow0}g(t)=\infty,\qquad u(x)\le g(d(x)), \qquad d(x):=\dist(x,\partial\Omega).
\]
\end{theorem}

Using the preceding estimate, Keller proved the existence of a boundary blow-up solution. Let $\Omega$ be a regular domain, for instance a Lipschitz domain. By the method of subsolutions and supersolutions (see, e.g., \cite[Sec.~9.3]{E}), for each $m\ge1$ there exists a unique solution $u_m\in C^2(\Omega) \cap C(\overline \Omega)$ of
\[
\left\{
\begin{aligned}
\Delta u_m &=f(u_m)&&\text{in }\;\Omega,\\
u_m&=m &&\text{on }\;\partial \Omega,
\end{aligned}
\right.
\]
The comparison principle gives
\[
u_m\le u_{m+1}\qquad\text{in }\;\Omega.
\]
By Theorem~\ref{thm2.5}, $u_m(x)\le g(d(x))$ uniformly in $m$. Let
\[
u(x):=\lim_{m\to\infty}u_m(x).
\]
By standard elliptic regularity (see, e.g., \cite{GT}), $u\in C^2(\Omega)$ and $\Delta u=f(u)$ in $\Omega$.

For every $m$, $u\ge u_m$ in $\Omega$ and $u_m=m$ on $\partial\Omega$, so $\liminf_{x\to\partial\Omega}u(x)\ge m$. Since $m$ is arbitrary, $u(x)\to\infty$ as $x\to\partial\Omega$. Thus $u$ is a boundary blow-up solution.

Moreover, if $v$ is any boundary blow-up solution, then the comparison principle gives $u_m\le v$ in $\Omega$ for every $m$, and hence $u\le v$.
We call $u$ the \textit{minimal boundary blow-up solution}.

Using the minimal boundary blow-up solution in a containing ball, we obtain a common lower bound for all boundary blow-up solutions.

\begin{lemma}[Lower bound]\label{lem2.3}
Let $\Omega\subset\mathbb R^n$ be a bounded domain. Choose a ball $B$ with $\Omega\Subset B$, and let $U_B$ be the minimal boundary blow-up solution in $B$. Set
\[
t_0:=\inf_{\Omega} U_B.
\]
Then every boundary blow-up solution $u$ in $\Omega$ satisfies
\[
u\ge U_B\quad\text{in }\;\Omega, \qquad\text{and hence}\qquad u\ge t_0 \quad\text{in }\;\Omega.
\]
\end{lemma}

\begin{proof}
Since $\Omega\Subset B$, the function $U_B$ is bounded on $\overline\Omega$, whereas $u(x)\to\infty$ as $x\to\partial\Omega$.
Thus
\[
\liminf_{x\to\partial\Omega}\,(u-U_B)(x)\ge0,
\]
and the comparison principle gives $u \ge U_B$ in $\Omega$.
\end{proof}

Loewner and Nirenberg \cite{LN} introduced another important solution, called a \textit{maximal solution}, which need not be a boundary blow-up solution but can be constructed in any bounded domain $\Omega$.
Let $\{\Omega_m\}_{m=1}^\infty$ be an exhaustion of $\Omega$ by smooth subdomains; that is,
\[
\Omega_m\Subset\Omega_{m+1}\Subset\Omega, \qquad \bigcup_{m=1}^\infty\Omega_m=\Omega.
\]
For each $m\ge1$, let $u_m$ be the minimal boundary blow-up solution in $\Omega_m$, and let $U$ be the minimal boundary blow-up solution in a ball containing $\overline\Omega$.
By Remark~\ref{rmk2.2}, $\{u_m\}_{m=1}^\infty$ is decreasing and bounded below by $U$.
Hence
\[
u:=\lim_{m\to\infty}u_m
\]
exists and solves $\Delta u = f(u)$ in $\Omega$.

Moreover, if $v$ satisfies $\Delta v = f(v)$ in $\Omega$, then the comparison principle gives $u_m\ge v$ for every $m$. Thus $u\ge v$, so $u$ is the maximal solution.

For $0< \alpha \le 1$, whenever the integral converges, set
\[
h_\alpha(t):=\int_t^\infty F(s)^{-\alpha/2}\,ds.
\]
For $\alpha=1$, we have $h_1=\sqrt{2}\,\psi$.
We will also assume that
\[
t\mapsto F(t)^{-\alpha/2}f(t)
\]
is nondecreasing.
The following observation shows that this monotonicity is preserved when the additive constant in $F$ is increased.

\begin{lemma}	\label{lem2.8}
If $F^{-\alpha/2}f$ is nondecreasing on an interval $I$, then $(F+C)^{-\alpha/2}f$ is nondecreasing on $I$ for every $C\ge0$.
\end{lemma}

\begin{proof}
The hypothesis implies that $f^{2/\alpha}/F= (F^{-\alpha/2}f)^{2/\alpha}$ is nondecreasing on $I$.
Since $F'=f\ge0$, $F$ is also nondecreasing.
For $C\ge0$, the function $s\mapsto s/(s+C)$ is nondecreasing on $(0,\infty)$, so $F/(F+C)$ is nondecreasing on $I$. Hence
\[
\frac{f^{2/\alpha}}{F+C} = \frac{f^{2/\alpha}}{F}\frac{F}{F+C}
\]
is nondecreasing, being the product of two nonnegative nondecreasing functions. Therefore
\[
(F+C)^{-\alpha/2}f = \left(\frac{f^{2/\alpha}}{F+C}\right)^{\alpha/2}
\]
is nondecreasing on $I$.
\end{proof}

\section{A uniqueness criterion via gradient asymptotics}\label{sec3}

Throughout the uniqueness argument, we impose the following assumption: for some $\alpha\in(0,1]$ and $t_0\in\mathbb R$,
\begin{equation}	\label{eq3.1}
\int_{t_0}^\infty F(t)^{-\alpha/2}\,dt<\infty,
\qquad
t\mapsto F(t)^{-\alpha/2}f(t)\quad\text{is nondecreasing on }\;[t_0,\infty).
\end{equation}
As shown in Section~\ref{sec2},  this condition is preserved when $F$ is replaced by $F+C$, $C\ge0$.
For $\alpha\in(0,1)$, the integrability condition in \eqref{eq3.1} is stronger than the Keller--Osserman condition, since
\[
F(t)^{-1/2}\le F(t)^{-\alpha/2}
\]
for all sufficiently large $t$.
Moreover, the monotonicity of $F^{-\alpha/2}f$ is equivalent to the convexity of $F^{1-\alpha/2}$ on $[t_0,\infty)$.

The first result treats the case in which the interval $[t_0,\infty)$ in \eqref{eq3.1} contains the range of every boundary blow-up solution. Its short proof captures the essential comparison argument.

\begin{theorem}[A special case]		\label{thm3.2}
Let $\Omega\subset\mathbb R^n$ be a bounded domain, and let $t_0$ be a common lower bound for all boundary blow-up solutions, for instance that given by Lemma~\ref{lem2.3}.
Assume that, for some $\alpha\in(0,1)$, condition \eqref{eq3.1} holds with this choice of $t_0$.
Let $u$ be a boundary blow-up solution of
\[
\Delta u=f(u)\quad\text{in }\Omega,
\]
such that
\begin{equation}\label{eq3.3}
\limsup_{x\to\partial\Omega}Q_u(x)\le1,\qquad Q_u(x)=\frac{|\nabla u(x)|^2}{2F(u(x))}.
\end{equation}
Then $u$ is the unique boundary blow-up solution in $\Omega$.
\end{theorem}

\begin{proof}
Since $\alpha\in(0,1)$, we have $1<1/\alpha$.
Thus \eqref{eq3.3} implies that $Q_u<1/\alpha$ in a sufficiently small neighborhood of $\partial\Omega$.
Hence there exists $\Omega_0\Subset\Omega$ such that
\[
Q_u<1/\alpha\qquad\text{in }\;\Omega\setminus\overline\Omega_0.
\]
Moreover, $Q_u$ is bounded on $\overline\Omega_0$.
By Lemma~\ref{lem2.8}, after increasing the additive constant in $F$, if necessary, and continuing to denote the resulting primitive by $F$, we may assume that $Q_u\le1$ on $\overline\Omega_0$, while the monotonicity in \eqref{eq3.1} is preserved.
Therefore
\begin{equation}\label{eq1953wed}
Q_u(x)<1/\alpha\qquad\text{throughout }\;\Omega.
\end{equation}

Set
\[
h(t):=\int_t^\infty F(s)^{-\alpha/2}\,ds,
\qquad
a(t):=F(t)^{-\alpha/2}f(t).
\]
Then $h$ is strictly decreasing and
\begin{equation}\label{eq3.5}
\Delta h(u)=a(u)(\alpha Q_u-1).
\end{equation}

Let $\tilde u$ be any other boundary blow-up solution.
Suppose that $u <\tilde u$ at some point in $\Omega$.
Then, for some $\epsilon>0$,
\[
D_\epsilon:=\{x\in\Omega:h(u)-h(\tilde u)>\epsilon\} \neq \emptyset.
\]
Since $h(u),\, h(\tilde u) \to 0$ as $x \to \partial\Omega$, we have $D_\epsilon \Subset\Omega$. Set
\[
w:=h(u)-h(\tilde u)-\epsilon.
\]
Then $w=0$ on $\partial D_\epsilon$ and $\tilde u>u$ in $D_\epsilon$.
By \eqref{eq3.5},
\begin{align*}
\Delta w
&=\alpha a(\tilde u)(Q_u-Q_{\tilde u})
 +(a(\tilde u)-a(u))(1-\alpha Q_u)\\
&=:I_1+I_2.
\end{align*}
The monotonicity of $a$ in $[t_0,\infty)$ and \eqref{eq1953wed} imply that $I_2 \ge 0$.
Moreover,
\begin{align*}
I_1 &= \frac{\alpha}{2}F(\tilde u)^{\alpha/2-1} f(\tilde u)
 \left(|\nabla h(u)|^2-|\nabla h(\tilde u)|^2\right)\\
&\qquad +\frac{\alpha}{2}F(\tilde u)^{\alpha/2-1}f(\tilde u)\, \frac{|\nabla u|^2}{F(u)}
\left(F(\tilde u)^{1-\alpha}-F(u)^{1-\alpha}\right).
\end{align*}
The second term is nonnegative since $F$ is nondecreasing and $1-\alpha>0$.
The first term is $\vec b\cdot\nabla w$, where
\[
\vec b:=\frac{\alpha}{2}F(\tilde u)^{\alpha/2-1}f(\tilde u) \nabla\{h(u)+h(\tilde u)\}.
\]
Hence
\[
\Delta w-\vec b\cdot\nabla w\ge0\qquad\text{in }D_\epsilon.
\]
Since $w=0$ on $\partial D_\epsilon$, the maximum principle gives $w\le0$ in $D_\epsilon$, contradicting its definition.
Therefore
\begin{equation}\label{eq3.6}
u\ge\tilde u\qquad\text{in }\;\Omega.
\end{equation}

Next, suppose that $\tilde u<u$ somewhere in $\Omega$.
Then, for some $\epsilon>0$,
\[
\emptyset \neq \tilde D_\epsilon:=\{x\in\Omega:h(\tilde u)-h(u)>\epsilon\} \Subset \Omega.
\]
Set
\[
\tilde w:=h(\tilde u)-h(u)-\epsilon.
\]
Then $\tilde u<u$ in $\tilde D_\epsilon$, and
\[
\Delta \tilde w =\alpha a(\tilde u)(Q_{\tilde u}-Q_u)  +(a(u)-a(\tilde u))(1-\alpha Q_u).
\]
As before, the second term is nonnegative in $\tilde D_\epsilon$.
We write the first term as
\begin{align*}
I_3 + I_4 &:=
\frac{\alpha}{2}\,F(\tilde u)^{\alpha/2-1}f(\tilde u)\left(|\nabla h(\tilde u)|^2-|\nabla h(u)|^2\right)\\
&\qquad +\frac{\alpha}{2}\,F(\tilde u)^{\alpha/2-1}f(\tilde u)\,\frac{\abs{\nabla u}^2}{F(u)}
\left(F(u)^{1-\alpha}-F(\tilde u)^{1-\alpha}\right).
\end{align*}
Again, $I_4\ge0$, while $I_3=\vec b\cdot\nabla\tilde w$ with the same $\vec b$ as above.
Hence
\[
\Delta \tilde w-\vec b\cdot\nabla \tilde w\ge0\qquad\text{in }\;\tilde D_\epsilon.
\]
The maximum principle again yields a contradiction.
Thus $\tilde u\ge u$ in $\Omega$.
Together with \eqref{eq3.6}, this gives $u=\tilde u$.
\end{proof}

We next extend it to the case where the monotonicity assumption holds only for sufficiently large values.
The proof is more technical but uses only standard ingredients and requires no regularity of $\partial\Omega$.
We first recall a smooth interior exhaustion of $\Omega$ whose boundaries approach $\partial\Omega$.

\begin{lemma} \label{lem3.7}
Let $\Omega\subset\mathbb R^n$ be a bounded domain, and let
$d(x)=\dist(x,\partial\Omega)$.
There exists an exhaustion of $\Omega$ by smooth domains $\Omega_j$ such that
\[
\Omega_j\Subset\Omega_{j+1}\Subset\Omega,
\qquad \bigcup_{j=1}^\infty\Omega_j=\Omega,
\]
and
\begin{equation}\label{eq3.8}
\sup_{x\in\partial\Omega_j}d(x) \to 0\qquad\text{as }\;j\to \infty.
\end{equation}
\end{lemma}

\begin{proof}
By the regularized distance construction of Lieberman \cite{Lieberman1985}, there exist a positive function $\rho\in C^\infty(\Omega)$ and constants $c_0,C_0>0$ such that
\[
c_0d(x)\le \rho(x)\le C_0d(x)\qquad\text{in }\Omega.
\]
Fix $x_0 \in\Omega$. By Sard's theorem, choose regular values
\[
0<\epsilon_j<\rho(x_0),\qquad \epsilon_j\downarrow0,
\]
of $\rho$, and let $\Omega_j$ be the connected component of $\{\rho>\epsilon_j\}$ containing $x_0$.
Since $d(x)>\epsilon_j/C_0$ in $\Omega_j$, we have $\Omega_j\Subset\Omega$.
Moreover,
$\rho\ge\epsilon_j>\epsilon_{j+1}$ on $\overline \Omega_j$, so $\overline \Omega_j \subset \Omega_{j+1}$.
Since $\epsilon_j$ is a regular value,
$\partial\Omega_j\subset \{\rho=\epsilon_j\}$ is smooth.
The domains $\Omega_j$ exhaust $\Omega$. Indeed, for any $x\in\Omega$, choose a path in $\Omega$ joining $x_0$ to $x$.
Its image is compact, so $\rho$ has a positive minimum on it.
For all sufficiently large $j$, the path lies in $\{\rho>\epsilon_j\}$, and hence $x\in\Omega_j$.
Finally, $d\asymp \rho=\epsilon_j$ on $\partial\Omega_j$, and therefore \eqref{eq3.8} holds.
\end{proof}

\begin{theorem}	\label{thm3.9}
Let $\Omega\subset\mathbb R^n$ be a bounded domain.
Assume that, for some $\alpha\in(0,1)$ and $t_0 \in \mathbb R$, condition \eqref{eq3.1} holds.
Let $u$ be a boundary
blow-up solution of
\[
\Delta u=f(u)\quad\text{in }\Omega,
\]
such that 
\begin{equation}\label{eq3.10}
\limsup_{x\to\partial\Omega}Q_u(x)\le1,\qquad Q_u(x)=\frac{|\nabla u(x)|^2}{2F(u(x))}.
\end{equation}
Then $u$ is the unique boundary blow-up solution in $\Omega$.
\end{theorem}

\begin{proof}
Let $\tilde u$ be any other boundary blow-up solution. We show that $u=\tilde u$.

Fix $t_1>t_0$.
Since $u,\,\tilde u \to \infty$ as $x \to \partial \Omega$ and $1<1/\alpha$, \eqref{eq3.10} gives a boundary neighborhood in which
\begin{equation}\label{eq3.11}
u,\tilde u\ge t_1, \qquad Q_u<1/\alpha
\end{equation}
Let $\{\Omega_j\}$ be the exhaustion in
Lemma~\ref{lem3.7}. 
Choosing $j_0$ sufficiently large and setting
\[
K:=\overline \Omega_{j_0},
\]
we may assume that \eqref{eq3.11} holds throughout
$\Omega\setminus  K$.
Enlarging $j_0$ if necessary, we may also arrange that
\begin{equation}\label{eq3.12}
C_n \diam(\Omega) f(t_1)\, \abs{\Omega\setminus K}^{1/n}\le t_1-t_0,
\end{equation}
where $C_n$ is the constant in the Alexandrov--Bakelman--Pucci estimate.

For $j>j_0$, set
\[
\tau_j^-:=\inf_{\partial\Omega_j}  \min (u,\tilde u),
\qquad
\tau_j^+:=\sup_{\partial\Omega_j} \max (u,\tilde u).
\]
By \eqref{eq3.8} and the boundary blow-up condition, $\tau_j^-\to\infty$.
After discarding finitely many indices, we may assume that
\[
\tau_j^-\ge t_1, \qquad \tau_j^+\ge\max_{\partial K} u.
\]
Let $V_j$ solve
\[
-\Delta V_j=1\quad\text{in}\; \Omega_j \setminus K,
\qquad V_j=0\quad\text{on }\partial (\Omega_j \setminus K).
\]
The maximum principle and the Alexandrov--Bakelman--Pucci estimate give
\[
0 \le V_j \le C_n\diam(\Omega) \,\abs{\Omega\setminus K}^{1/n}.
\]
Hence, by \eqref{eq3.12}, $v_j:=t_1-f(t_1)V_j$ satisfies
\[
t_0\le  v_j\le t_1\quad\text{in }\;\Omega_j \setminus K,
\]
with  $v_j=t_1$ on $\partial (\Omega_j \setminus K)$.
Moreover,
\[
\Delta  v_j=f(t_1)\ge f(v_j).
\]

The method of subsolutions and supersolutions yields solutions $u_j^\pm$ of
\[
\left\{
\begin{aligned}
\Delta u_j^\pm&=f(u_j^\pm)&&\text{in }\;\Omega_j \setminus K,\\
u_j^\pm&=u&&\text{on }\;\partial K,\\
u_j^\pm&=\tau_j^\pm&&\text{on }\; \partial \Omega_j.
\end{aligned}
\right.
\]
Comparison with $v_j$ and $u$ gives
\begin{equation}\label{eq3.15}
t_0\le u_j^-\le u\le u_j^+
\qquad\text{in }\;\Omega_j \setminus K.
\end{equation}
Set
\[
h(t):=\int_t^\infty F(s)^{-\alpha/2}\,ds,
\qquad
a(t):=F(t)^{-\alpha/2}f(t).
\]
All functions in \eqref{eq3.15} take values in $[t_0,\infty)$, where $a$ is nondecreasing, and  $Q_u<1/\alpha$ in $\Omega_j \setminus K$.
The decompositions used in the proof of Theorem~\ref{thm3.2} therefore give continuous vector fields $\vec b_j^\pm$ such that
\begin{align*}
\Delta \left\{ h(u_j^-)-h(u) \right\}  -\vec b_j^-\cdot\nabla \left\{ h(u_j^-)-h(u)\right) &\ge 0,\\
\Delta \left\{ h(u)-h(u_j^+) \right\} -\vec b_j^+\cdot\nabla \left\{ h(u)-h(u_j^+)\right\} &\ge0
\end{align*}
in $\Omega_j \setminus K$.
Both differences are nonnegative and vanish on $\partial K$.
On $\partial\Omega_j$, they are bounded above by $h(\tau_j^-)$.
The maximum principle therefore gives
\[
0\le h(u_j^-)-h(u)\le h(\tau_j^-), \qquad
0\le h(u)-h(u_j^+)\le h(\tau_j^-) \quad\text{in }\;\Omega_j \setminus K.
\]
Since $h(\tau_j^-)\to0$ and $h$ is strictly decreasing, it follows that
\begin{equation}\label{eq3.16}
u_j^-, \,u_j^+\to u \qquad\text{locally uniformly in }\;\Omega\setminus  K.
\end{equation}

On the other hand, set $w_j:=\tilde u-u_j^-$. Then
\[
\Delta w_j = f(\tilde u)-f(u_j^-)=q_j(x) w_j \quad\text{in }\;\Omega_j \setminus K,
\]
where $q_j \ge0$ and $q_j \in L^\infty(\Omega_j\setminus K)$ by the monotonicity and local Lipschitz continuity of $f$.
Moreover,
\[
w_j= \tilde u - u \quad\text{on }\;\partial K,
\qquad
w_j\ge0\quad\text{on }\;\partial \Omega_j.
\]
Hence the maximum principle gives
\[
\tilde u-u_j^- \ge\min\left\{0,\,\min_{\partial K}\, (\tilde u-u) \right\} \quad\text{in }\;\Omega_j \setminus K.
\]
Similarly, applying the maximum principle to $\tilde u-u_j^+$ yields
\[
\tilde u-u_j^+ \le\max\left\{0,\, \max_{\partial K}\, (\tilde u-u) \right\}
\quad\text{in }\;\Omega_j \setminus K.
\]
Letting $j\to\infty$ and using \eqref{eq3.16}, we obtain
\begin{equation}\label{eq3.17}
\min\left\{0,\,\min_{\partial K}\,(\tilde u-u)\right\} \le \tilde u-u \le \max\left\{0,\,\max_{\partial K}\,(\tilde u-u)\right\}
\quad\text{in }\;\Omega\setminus  K.
\end{equation}

Recall that $K=\overline \Omega_{j_0}$. In  $\Omega_{j_0}$, the difference $w:=\tilde u -u$ satisfies
\[
\Delta w=q(x)w,
\qquad q(x)\ge0.
\]
Applying the maximum principle to $w$ and $-w$ yields the same two-sided estimate as in \eqref{eq3.17}. Hence \eqref{eq3.17} holds throughout $\Omega$.
Suppose $w\not\equiv0$.
Then $w$ attains either a positive global maximum or a negative global minimum on $\partial K\subset\Omega$.
The strong maximum principle then implies that $w$ is constant.
Thus $\tilde u=u+c$ for some $c\ne0$. Since $u$ and $\tilde u$ satisfy the same equation,
\[
f(u+c)=f(u)\quad\text{in }\;\Omega.
\]
The range of $u$ contains all sufficiently large values, so $f(t+c)=f(t)$ for all sufficiently large $t$.
Since $f$ is nondecreasing, this forces $f$ to be eventually constant, contradicting the Keller--Osserman condition. Therefore $c=0$, and hence $\tilde u=u$.
\end{proof}

\begin{remark}
The boundary condition in Theorems~\ref{thm3.2} and ~\ref{thm3.9} can be weakened to
\[
\limsup_{x\to\partial\Omega} Q_u(x)<\frac1\alpha.
\]
Indeed, the proofs only require $Q_u<1/\alpha$ in a neighborhood of $\partial\Omega$.
\end{remark}

\section{Boundary gradient asymptotics and uniqueness in $C^{1,1}$ domains}	\label{sec4}

We begin with the one-dimensional model problem
\begin{equation}	\label{eq4.1}
\left\{
\begin{aligned}
u''(r)&=f(u(r)), && r>0,\\
u(r)&\to\infty, && r\downarrow0.
\end{aligned}
\right.
\end{equation}
Let $F$ be any primitive of $f$.
Integrating \eqref{eq4.1} once gives
\[
u'(r)^2=2F(u(r))+c,
\]
for some constant $c$.
Since $u$ is convex and $u(r)\to\infty$ as $r\downarrow0$, we have $u'<0$ on some interval $(0,r_0)$. Thus
\[
-\frac{dr}{du}=\frac{1}{\sqrt{2F(u)+c}}, \qquad
r=\int_u^\infty \frac{ds}{\sqrt{2F(s)+c}}=:\psi_c(u).
\]
The Keller--Osserman condition \eqref{eq1.1} ensures that $\psi_c$ is finite. Define
\begin{equation}	\label{eq4.2}
\psi(t):=\psi_0(t)=\int_t^\infty \frac{ds}{\sqrt{2F(s)}}.
\end{equation}
Since $F(t)\to\infty$ as $t\to\infty$, comparison of the tail integrals gives
\begin{equation}	\label{eq4.3}
\lim_{t\to\infty} \,\frac{\psi(t)}{\psi_c(t)}=1.
\end{equation}
Since $r=\psi_c(u(r))$, \eqref{eq4.3} gives
\[
\frac{\psi(u(r))}{r}\to1 \qquad\text{as }\;r\downarrow0.
\]

We next consider the asymptotic behavior of solutions of
\begin{equation}	\label{eq4.4}
\left\{
\begin{aligned}
u''(r)+A(r)u'(r)&=f(u(r)), && r>0,\\
u(r)&\to\infty, && r\downarrow0.
\end{aligned}
\right.
\end{equation}
Assume that, for some $r_0,\,M>0$,
\begin{equation}	\label{eq4.6}
\abs{A(r)}\le M\quad\text{for }\; 0<r\le r_0.
\end{equation}

The following lemma is well-known; see \cite{BE, Ma}.
We provide a self-contained proof for completeness.
\begin{lemma}	\label{lem4.6}
Let $u$ be a solution of \eqref{eq4.4}, and assume that $A$ satisfies \eqref{eq4.6}. Then
\begin{equation}	\label{eq4.7}
\frac{\psi(u(r))}{r}\to1 \qquad\text{as }\;r\downarrow0,
\end{equation}
where $\psi$ is defined by \eqref{eq4.2}.
\end{lemma}

\begin{proof}
First, $u'<0$ on some interval $(0,r_1)$, with $r_1\le r_0$.
Indeed,
\[
\frac{d}{dr}\left(e^{\mu(r)}u'(r) \right)
=e^{\mu(r)}f(u(r))
\ge 0,\qquad \mu(r)=\int_{r_1}^r A(t)\,dt,
\]
so $u'$ can change sign at most once.
The condition $u(r)\to\infty$ as $r\downarrow0$ then implies that $u'<0$ near $0$.
Multiplying the equation by $2u'$, we obtain
\[
\frac{d}{dr}(u')^2+2A(u')^2 =\frac{d}{dr} (2F(u)).
\]
Fix $\delta \in(0,r_1)$ and set
\[
B(r):=\exp\left(-2\int_r^\delta A(t)\,dt\right).
\]
Then
\[
\frac{d}{dr} (B(u')^2) = B\frac{d}{dr}(2F(u)).
\]
For $0<r<\delta$,
\[
e^{-2M\delta}\le B(r)\le e^{2M\delta},
\]
while
\[
\frac{d}{dr}F(u)=f(u)u'\le0.
\]
Consequently,
\[
e^{2M\delta }\frac{d}{dr}(2F(u)) \le \frac{d}{dr}\left(B(u')^2\right) \le e^{-2M\delta}\frac{d}{dr}(2F(u)).
\]
Integrating from $r$ to $\delta$, we obtain
\[
e^{-2M\delta}(2F(u)+c_1) \le B(u')^2 \le e^{2M\delta}(2F(u)+c_2)
\]
for some constants $c_1$, $c_2$. Hence
\[
e^{-4M\delta}(2F(u)+c_1) \le (u')^2 \le e^{4M\delta}(2F(u)+c_2).
\]
Since $u'<0$,
\[
\frac{e^{-2M\delta}}{\sqrt{2F(u)+c_2}} \le -\frac{dr}{du} \le \frac{e^{2M\delta}}{\sqrt{2F(u)+c_1}}.
\]
Integrating from $u(r)$ to $\infty$ gives
\[
e^{-2M\delta}\psi_{c_2}(u(r)) \le r\le e^{2M\delta}\psi_{c_1}(u(r)).
\]
Using \eqref{eq4.3}, we conclude that
\[
e^{-2M\delta} \le \liminf_{r\downarrow0}\frac{\psi(u(r))}{r} \le \limsup_{r\downarrow0}\frac{\psi(u(r))}{r} \le e^{2M\delta}.
\]
Letting $\delta\downarrow 0$ proves \eqref{eq4.7}.
\end{proof}

We use the ODE solutions as barriers to derive boundary gradient estimates for blow-up solutions.

\begin{theorem}\label{thm4.8}
Let $\Omega\subset\mathbb R^n$ be a bounded $C^{1,1}$ domain, and let $u$ be a boundary blow-up solution of
\[
\Delta u=f(u)\quad\text{in }\;\Omega.
\]
Assume that
\begin{equation}	\label{cond_growth}
\limsup_{t\to\infty} \psi(t) \,\frac{f(t)}{\sqrt{F(t)}}<\infty,\qquad \psi(t)=\int_t^\infty \frac{ds}{\sqrt{2F(s)}}.
\end{equation}
Then $v:=\psi(u)$ extends to a function in $C^1(\overline\Omega)$ and satisfies
\[
v=0,\qquad \nabla v=\vec n\quad\text{on }\partial\Omega,
\]
where $\vec n$ denotes the inward unit normal. Consequently,
\[
Q_u(x)=\abs{\nabla v(x)}^2 \to 1 \qquad\text{as }\;x\to\partial\Omega.
\]
\end{theorem}

\begin{proof}
By increasing the additive constant in $F$, if necessary, we may assume that $F>0$ on the range of $u$.
This does not affect condition~\eqref{cond_growth}.
Let
\[
d(x):=\operatorname{dist}(x,\partial\Omega),\qquad
\Omega_{\delta}:=\{x\in\Omega: d(x)<\delta\}.
\]
Since $\Omega$ is $C^{1,1}$, the signed distance to $\partial\Omega$, taken positive in $\Omega$, is $C^{1,1}$ in a tubular neighborhood
of $\partial\Omega$.
Hence, for some $\delta_0>0$, $d$ is $C^{1,1}$ in $\Omega_{\delta_0}$, each $x\in\Omega_{\delta_0}$ has a unique nearest point $\pi(x)\in\partial\Omega$, and
\[
\nabla d(x)=\vec n(\pi(x)).
\]
Moreover, there exists $R>0$ such that, for every $z\in\partial\Omega$,
\[
B_R(z+R \vec n(z))\subset\Omega,\qquad
B_R(z-R\vec n(z))\subset\mathbb R^n\setminus \overline \Omega,
\]
and both balls are tangent to $\partial\Omega$ at $z$.
Let $U_i$ be the minimal boundary blow-up solution in the interior tangent ball.
It is radial, and, writing it as a function of the distance $r$ to the boundary of the ball,
\[
U_i''(r)-\frac{n-1}{R-r}\,U_i'(r)=f(U_i(r)),	 \qquad 0<r<R,
\]
with
\[
U_i(r)\to\infty\quad\text{as }\;r\downarrow0, \qquad U_i'(R)=0.
\]
Let $U_e$ be an exterior radial comparison solution satisfying
\[
U_e''(r)+\frac{n-1}{R+r}\,U_e'(r)=f(U_e(r)), \qquad 0<r<r_0,
\]
with
\[
U_e(r)\to\infty\quad\text{as }r\downarrow0, \qquad U_e(r_0)<\inf_\Omega u,
\]
where $r_0>0$ is independent of $z$.
Such a solution is obtained from a radial solution in the annulus that blows up on $\partial B_R$ and takes a constant value less than $\inf_\Omega u$ on $\partial B_{R+r_0}$.
Applying the comparison principle first to slightly perturbed tangent balls, whose boundaries do not meet $\partial\Omega$, and then letting the perturbation tend to zero, we obtain
\begin{equation}	\label{eq2212tue}
U_e(r)\le u(z+r\vec n(z))\le U_i(r)
\end{equation}
for all $z\in\partial\Omega$ and all sufficiently small $r>0$, uniformly in $z$.
Set
\[
V_i(r):=\psi(U_i(r)),
\qquad
V_e(r):=\psi(U_e(r)).
\]
By Lemma~\ref{lem4.6},
\[
\frac{V_i(r)}r\to1,\qquad \frac{V_e(r)}r\to1 \qquad\text{as }\;r\downarrow0.
\]
Because $\psi$ is decreasing, \eqref{eq2212tue} gives
\[
V_i(r)\le v(z+r\vec n(z))\le V_e(r).
\]
For $x$ sufficiently close to $\partial\Omega$, taking $z=\pi(x)$ and $r=d(x)$ yields
\begin{equation}	\label{eq4.10}
v(x)=d(x)+o(d(x))
\end{equation}
uniformly as $d(x)\to0$.
Hence $v$ extends continuously to $\overline\Omega$ by setting $v=0$ on $\partial\Omega$.

Let $x_k\in\Omega$ satisfy
\[
r_k:=d(x_k)\to0, \qquad z_k:=\pi(x_k).
\]
Choose an orthogonal transformation $\mathsf R_k$ such that
\[
\mathsf R_k\vec e_1=\vec n(z_k)=\nabla d(x_k),
\]
and, for $y\in B_{1/2}$, define
\[
X_k(y):=x_k+r_k\mathsf R_k y, \qquad v_k(y):=\frac{v(X_k(y))}{r_k}, \qquad d_k(y):=\frac{d(X_k(y))}{r_k}.
\]
Since $d$ is $1$-Lipschitz,
\[
\frac{r_k}{2}\le d(X_k(y))\le\frac{3r_k}{2},\qquad\text{for }y\in B_{1/2}.
\]
Hence \eqref{eq4.10} yields
\[
\norm{v_k-d_k}_{L^\infty(B_{1/2})} \to 0.
\]
The $C^{1,1}$ regularity of $d$ gives
\[
d_k(y)=1+y\cdot\vec e_1+O(r_k)
\]
uniformly in $B_{1/2}$.
Hence
\begin{equation}\label{eq4.11}
v_k \to \ell \quad\text{uniformly in }\;B_{1/2}, \qquad \ell(y):=1+y\cdot\vec e_1.
\end{equation}

Let $\phi=\psi^{-1}$. Then
\[
\Delta v=a(v)(\abs{\nabla v}^2-1), \qquad a(s):=\frac{f(\phi(s))}{\sqrt{2F(\phi(s))}},
\]
and hence
\[
\Delta v_k=b_k(v_k)(\abs{\nabla v_k}^2-1), \qquad b_k(s):=r_k a(r_ks).
\]
Condition~\eqref{cond_growth} implies 
\[
s a(s)\le C
\]
for all sufficiently small $s>0$.
By \eqref{eq4.11}, for all sufficiently large $k$,
\[
1/4 \le v_k\le2 \quad\text{in }\;B_{1/2}.
\]
Therefore, after enlarging $C$,
\[
0\le b_k(s)\le C, \qquad 1/4 \le s\le2.
\]
Define
\[
h_k(s):= \int_1^s \exp\left(-\int_1^\tau b_k(\rho)\,d\rho\right)d\tau.
\]
Then, on $[1/4,2]$,
\[
h_k''=-b_kh_k', \qquad C^{-1}\le h_k'\le C, \qquad \abs{h_k''} \le C.
\]
Set $w_k:=h_k(v_k)$. Using the equation for $v_k$, 
\[
\Delta w_k =h_k''(v_k)|\nabla v_k|^2+h_k'(v_k)\Delta v_k =-b_k(v_k)h_k'(v_k).
\]
Hence
\[
\abs{\Delta w_k} \le C \qquad\text{in }\;B_{1/2}.
\]
Since $1/4 \le v_k \le 2$ in $B_{1/2}$, $h_k(1)=0$, and $h_k'$ is uniformly bounded, 
\[
\abs{w_k}=\abs{h_k(v_k)} \le C \qquad\text{in }\;B_{1/2}.
\]
Interior $W^{2,p}$ estimates, with $p>n$, and Sobolev embedding yield, after passing to a subsequence,
\[
w_k\to w\quad\text{in }C^{1,\gamma}(B_{1/4})
\]
for some $0<\gamma<1-n/p$.
Passing to a further subsequence,  by Arzel\`a--Ascoli gives $h_k\to h$ in $C^1([1/4,2])$.
By \eqref{eq4.11}, $w=h(\ell)$, and hence
\[
\nabla w_k(0)\to h'(1)\vec e_1=\vec e_1,
\]
since $h_k'(1)=1$, so $h'(1)=1$.
Moreover, \eqref{eq4.10} gives
\[
v_k(0)=\frac{v(x_k)}{r_k}\to1.
\]
Since $h_k'$ are uniformly Lipschitz and $h_k'(1)=1$,
\[
h_k'(v_k(0))\to1.
\]
Thus, from
\[
\nabla w_k(0)=h_k'(v_k(0))\nabla v_k(0),
\]
we obtain 
\[
\nabla v_k(0)\to\vec e_1. 
\]
Since every subsequence admits a further subsequence with the same limit, this convergence holds for the full sequence.
Since
\[
\nabla v_k(0)=\mathsf R_k^\top\nabla v(x_k), \qquad
\vec e_1=\mathsf R_k^\top\vec n(z_k),\qquad z_k=\pi(x_k),
\]
we obtain
\[
\nabla v(x)-\vec n(\pi(x))\to0 \qquad\text{as }\;d(x)\to0.
\]
For $z\in\partial\Omega$, the $C^{1,1}$ regularity of the boundary gives the expansion
\[
d(x)=\vec n(z)\cdot(x-z)+O(\abs{x-z}^2) \qquad\text{as }\;x\to z,\quad x\in\Omega.
\]
Since $d(x)\le |x-z|$, \eqref{eq4.10} and the preceding expansion yield
\[
v(x)=\vec n(z)\cdot(x-z)+o(\abs{x-z}).
\]
Thus $v$ is differentiable at $z$, with $\nabla v(z)=\vec n(z)$.
Since $\pi(x)\to z$ as $x\to z$ and $\vec n$ is continuous on $\partial\Omega$,
\[
\abs{\nabla v(x)-\vec n(z)} \le \abs{\nabla v(x)-\vec n(\pi(x))}  + \abs{\vec n(\pi(x))-\vec n(z)} \to 0.
\]
Thus $\nabla v$ extends continuously to $\partial\Omega$, with
$\nabla v(z)=\vec n(z)$, and hence $v\in C^1(\overline\Omega)$.
\end{proof}

\begin{corollary}	
Let $\Omega\subset\mathbb R^n$ be a bounded $C^{1,1}$ domain.
Assume that \eqref{cond_growth} holds and that, for some $\alpha\in(0,1)$ and $t_0\in\mathbb R$, condition \eqref{eq3.1} is satisfied.
Then $\Delta u=f(u)$ admits a unique boundary blow-up solution in $\Omega$.
\end{corollary}

\begin{proof}
Since $\Omega$ is $C^{1,1}$, it admits a minimal boundary blow-up solution $u$.
By Theorem~\ref{thm4.8},
\[
\lim_{x\to \partial \Omega} Q_u(x) = 1.
\]
The conclusion follows from Theorem~\ref{thm3.9}.
\end{proof}

\section{Uniqueness in convex domains}	\label{sec5}

At $\alpha=1$, the boundary asymptotic $Q_u\to1$ does not determine the sign of $1-Q_u$.
In convex domains, however, a different argument gives $Q_u\le1$ in $\Omega$, restoring the comparison argument when $t\mapsto F(t)^{-1/2} f(t)$ is nondecreasing, equivalently when $\sqrt F$ is convex.
We first assume this convexity on the full range of the minimal boundary blow-up solution. This stronger assumption keeps the proof short while retaining the essential comparison argument.
When the boundary is smooth, the argument below uses only the mean-convexity condition $H \ge 0$; convexity is used to pass to nonsmooth domains by smooth convex approximation.

\begin{theorem}[A special case]	\label{thm5.1}
Let $\Omega$ be a bounded convex domain, and let $u$ be the minimal boundary blow-up solution of
\[
\Delta u = f(u)\quad\text{in }\;\Omega.
\]
Then one can choose a primitive $F$ of $f$ such that
\[
Q_{u}(x) = \frac{\abs{\nabla u(x)}^2}{2F(u(x))} \le 1 \qquad\text{for all }\;x\in\Omega.
\]
If, moreover, $\sqrt F$ is convex on the range of $u$, then $u$ is the unique boundary blow-up solution.
\end{theorem}

\begin{proof}
For each $m \in \mathbb N$, let $u_m$ solve
\[
\Delta u_m=f(u_m)\quad\text{in }\;\Omega, \qquad
u_m=m\quad\text{on }\;\partial\Omega.
\]
Then $u_m\uparrow u$ locally in $C^2_{\rm loc}(\Omega)$ as $m\to\infty$.
By adding a nonnegative constant to $F$ if necessary, we may assume that $F>0$ on the range of $u_m$ for every sufficiently large $m$.
By Lemma~\ref{lem2.8}, this does not affect the monotonicity of $t\mapsto F(t)^{-1/2}f(t)$.

For each fixed $m$, approximate $f$ locally uniformly by smooth, nonnegative, nondecreasing functions, and exhaust $\Omega$ from within by smooth convex domains.
By standard elliptic estimates and stability, the corresponding solutions converge to $u_m$ in $C^1_{\mathrm{loc}}(\Omega)$. Hence it suffices to prove the estimate when $f$, $\Omega$, and $u_m$ are smooth.
Set
\[
P_m:=\frac12 \,\abs{\nabla u_m}^2-F(u_m).
\]
On the set  $\{P_m>0\}$, we have $\nabla u_m\ne 0$ and
\[
\Delta P_m -\frac{2f(u_m)}{\abs{\nabla u_m}^2}\, \nabla u_m\cdot\nabla P_m \ge 0.
\]
This is a standard calculation; see, for example, \cite{PSS}.
Applying the maximum principle on each connected component of $\{P_m>0\}$, we see that any positive maximum of $P_m$ must be attained on $\partial\Omega$.
Let $\nu$ denote the outward unit normal and $H=\operatorname{div} \nu$ the mean curvature of $\partial\Omega$. Since $u_m$ is constant on $\partial\Omega$,
\[
\Delta u_m=(u_m)_{\nu\nu}+H (u_m)_\nu \quad\text{on }\partial\Omega.
\]
Consequently,
\[
\partial_\nu P_m =(u_m)_\nu \left((u_m)_{\nu\nu}-f(u_m)\right) =-H (u_m)_\nu^2 \le 0,
\]
since a smooth convex domain has $H \ge 0$.
The Hopf lemma therefore rules out a positive boundary maximum of $P_m$. Hence
\[
\abs{\nabla u_m}^2 \le 2F(u_m) \quad\text{in }\;\Omega.
\]
Passing first to the limit in the smooth approximations and then letting $m\to\infty$, we obtain $Q_u\le1$ in $\Omega$.

It remains to prove uniqueness.
Let $\tilde u$ be any other boundary blow-up solution. 
By the minimality of $u$, we have $u \le \tilde u$ in $\Omega$. Set
\[
v:=\psi(u), \qquad \tilde v:=\psi(\tilde u),\qquad \psi(t)=\int_t^\infty \frac{ds}{\sqrt{2F(s)}}.
\]
Suppose, for contradiction, that $u<\tilde u$ at some point in $\Omega$.
Since $\psi$ is strictly decreasing, there exists $\epsilon>0$ such that
\[
D_\epsilon := \{x\in\Omega: v-\tilde v>\epsilon\} \ne\emptyset.
\]
Define  $w:= v- \tilde v-\epsilon$.
Since $v-\tilde v\to0$ as $x \to \partial\Omega$, we have $D_\epsilon \Subset\Omega$ and $w=0$ on $\partial D_\epsilon$.
Moreover, $\tilde u>u$ in $D_\epsilon$. Set
\[
a(t):=\frac{f(t)}{\sqrt{2F(t)}}.
\]
Using the monotonicity of $a$,  together with $\abs{\nabla v}^2=Q_u\le1$, we obtain
\begin{equation}	\label{eq1527thu}
\begin{aligned}
\Delta w &= a(u)(\abs{\nabla v}^2-1) -a(\tilde u)(\abs{\nabla \tilde v}^2-1)\\
&= a(\tilde u)(\abs{\nabla v}^2-\abs{\nabla \tilde v}^2)+(a(\tilde u)-a(u)) (1-\abs{\nabla v}^2)\\
&\ge a(\tilde u)\nabla(v+\tilde v)\cdot\nabla w
\end{aligned}
\end{equation}
in $D_\epsilon$.
The maximum principle therefore yields $w\le 0$ in $D_\epsilon$, contradicting the definition of $D_\epsilon$.
Hence $u=\tilde u$.
\end{proof}

The preceding gradient estimate also holds on smooth mean-convex domains, since its boundary argument requires only $H \ge 0$.
We next relax the assumption on $\sqrt F$, requiring convexity only on $[t_0,\infty)$.

\begin{theorem}	\label{thm5.3}
Let $\Omega\subset\mathbb R^n$ be a bounded convex domain. If $\sqrt F$ is convex on $[t_0,\infty)$ for some $t_0\in\mathbb R$, then
\[
\Delta u = f(u)\quad\text{in }\;\Omega
\]
admits a unique boundary blow-up solution.
\end{theorem}

\begin{proof}
Since $\Omega$ is convex, there exists the minimal boundary blow-up solution $u$.
Let $\tilde u$ be any other boundary blow-up solution.
We show that $u=\tilde u$.

After increasing $t_0$, if necessary, we may assume $f(t_0)>0$. 
Fix $t_1>t_0$.
Since $u(x)\to\infty$ as $x\to\partial\Omega$, there exists a smooth convex domain $\Omega_0 \Subset\Omega$ such that $\tilde u\ge u\ge t_1$ in $\Omega \setminus \Omega_0$.
We choose $\Omega_0$ sufficiently close to $\partial\Omega$ that
\begin{equation}	\label{eq1418thu}
\sup_{x \in \Omega} \,\dist(x, \Omega_0) \le \delta:=\sqrt{\frac{2(t_1-t_0)}{f(t_1)}}>0.
\end{equation}
Choose a smooth convex exhaustion
\[
\Omega_0 \Subset\Omega_j\Subset\Omega_{j+1}\Subset\Omega, \qquad\bigcup_{j=1}^{\infty}\Omega_j=\Omega.
\]
After discarding finitely many terms, we may assume that
\[
m_j:=\min_{\partial\Omega_j}u \ge t_1.
\]
For $v=u$ or $v=\tilde u$, let $v_j$ solve
\begin{equation}\label{eq5.5}
\left\{
\begin{aligned}
\Delta v_j&=f(v_j)&&\text{in }\;\Omega_j\setminus \overline \Omega_0,\\
v_j&=v&&\text{on }\;\partial \Omega_0,\\
v_j&=m_j&&\text{on }\;\partial\Omega_j.
\end{aligned}
\right.
\end{equation}
Since $u \le v$, the comparison principle gives
\begin{equation}\label{eq5.6}
v_j\le v\qquad\text{in }\;\Omega_j\setminus \overline \Omega_0.
\end{equation}

We next show that $v_j \ge t_0$.
Define
\[
\underline v(x) := t_0+\frac{t_1-t_0}{\delta^2}\,\rho(x)^2,\qquad \text{where }\;\rho(x):=\dist(x, \Omega_0).
\]
By \eqref{eq1418thu} and definition of $\rho(x)$, we have $t_0 \le\underline v\le t_1$ in $\Omega\setminus\overline \Omega_0$ and 
\[
\underline v=t_0 \le v \quad\text{on }\;\partial \Omega_0,
\qquad \underline v\le t_1 \le m_j \quad\text{on }\;\partial\Omega_j.
\]
Moreover, since $\abs{\nabla\rho}=1$ and $\Delta\rho\ge 0$ in $\mathbb R^n\setminus\overline \Omega_0$, we have
\[
\Delta\underline v = \frac{2(t_1-t_0)}{\delta^2} \left(\abs{\nabla \rho}^2+\rho\Delta\rho\right) \ge \frac{2(t_1-t_0)}{\delta^2}\quad\text{in }\;\Omega \setminus\overline \Omega_0.
\]
Since $f$ is nondecreasing and $\underline v\le t_1$, it follows that
\[
\Delta\underline v\ge f(t_1)\ge f(\underline v).
\]
Thus $\underline v$ is a subsolution of \eqref{eq5.5}. The
comparison principle yields
\begin{equation}\label{eq5.7}
v_j\ge\underline v\ge t_0 \qquad\text{in }\;\Omega_j\setminus\overline \Omega_0.
\end{equation}

We now establish a uniform estimate for
\[
P_j:=\frac12 \abs{\nabla v_j}^2-F(v_j)\quad \text{in }\;\Omega_j\setminus\overline \Omega_0.
\]
By the argument used in the proof of Theorem~\ref{thm5.1},
\begin{equation}\label{eq5.8}
P_j\le\max\left\{0,\,\sup_{\partial \Omega_0} P_j\right\} \qquad\text{in }\;\Omega_j\setminus\overline \Omega_0.
\end{equation}
The right-hand side of \eqref{eq5.8} is bounded independently of $j$.
Indeed, in a fixed neighborhood of $\partial \Omega_0$ we have, by \eqref{eq5.6} and \eqref{eq5.7},
\[
t_0\le v_j\le v,
\]
while the boundary value of $v_j$ on $\partial \Omega_0$ is the fixed function $v$.
Standard local boundary estimates therefore yield
\[
\sup_j \,\norm{\nabla v_j}_{L^\infty(\partial \Omega_0)}<\infty.
\]
Applying this argument to the families corresponding to both $v=u$ and $v=\tilde u$, we may choose a single constant $C\ge0$ such that
\[
\frac12\,\abs{\nabla v_j}^2\le F(v_j)+C \qquad\text{in }\;\Omega_j\setminus\overline \Omega_0.
\]
Replacing $F+C$ by $F$, we obtain
\begin{equation}\label{eq5.9}
\frac12\,\abs{\nabla v_j}^2\le F(v_j) \qquad\text{in }\;\Omega_j\setminus\overline \Omega_0
\end{equation}
for both families of solutions.
Set
\[
V_j:=\psi(v_j), \qquad V:=\psi(v),\qquad \psi(t)=\int_t^\infty \frac{ds}{\sqrt{2F(s)}},\qquad a(t):=\frac{f(t)}{\sqrt{2F(t)}}.
\]
Since $\psi$ is decreasing, \eqref{eq5.6} gives $V_j\ge V$. 
As in \eqref{eq1527thu},
\begin{align*}
\Delta(V_j-V) &= a(v_j)(\abs{\nabla V_j}^2-1) - a(v)(\abs{\nabla V}^2-1) \\ 
&= a(v) (\abs{\nabla V_j}^2-\abs{\nabla V}^2)
+(a(v)- a(v_j)) (1-\abs{\nabla V_j}^2).
\end{align*}
By \eqref{eq5.7}, \eqref{eq5.6}, the monotonicity of $a$, and \eqref{eq5.9}, the last term is nonnegative in $\Omega_j\setminus\overline \Omega_0$.
Hence
\[
\Delta(V_j-V)- a(v)\nabla(V_j+V)\cdot\nabla(V_j-V) \ge0 \quad\text{in }\;\Omega_j\setminus\overline \Omega_0.
\]
On $\partial \Omega_0$ we have $V_j=V$.
On $\partial\Omega_j$, we have $V_j=\psi(m_j)$ and $0\le V\le\psi(m_j)$, since $v\ge m_j$ there.
The maximum principle therefore gives
\[
0\le V_j-V\le\psi(m_j) \qquad\text{in }\;\Omega_j\setminus\overline \Omega_0.
\]
Since $m_j\to\infty$, we have $\psi(m_j)\to0$. By the strict monotonicity of $\psi$,
\begin{equation}\label{eq5.10}
v_j \to v \qquad\text{locally uniformly in }\;\Omega \setminus \overline \Omega_0.
\end{equation}

Let $u_j$ and $\tilde u_j$ denote the solutions of \eqref{eq5.5} corresponding to $u$ and $\tilde u$, respectively.
They have the same boundary value $m_j$ on $\partial\Omega_j$, while $u\le \tilde u$ on $\partial \Omega_0$.
The comparison principle therefore gives
\[
u_j\le \tilde u_j \qquad\text{in }\;\Omega_j\setminus\overline \Omega_0.
\]
Set $w_j:=\tilde u_j-u_j$.
Then
\[
\Delta w_j=q_j(x) w_j \qquad\text{in }\Omega_j\setminus\overline \Omega_0,
\]
where $q_j(x) \ge0$ is bounded. Moreover,
\[
w_j=0\quad\text{on }\partial\Omega_j, \qquad w_j=\tilde u-u \quad\text{on }\partial \Omega_0.
\]
The maximum principle yields
\[
0\le w_j\le\max_{\partial \Omega_0}\,(\tilde u -u) \qquad\text{in }\;\Omega_j\setminus\overline \Omega_0.
\]
Letting $j\to\infty$ and using \eqref{eq5.10}, we obtain
\begin{equation}\label{eq5.11}
0\le \tilde u -u\le\max_{\partial \Omega_0}\,(\tilde u-u) \qquad\text{in }\;\Omega \setminus\overline \Omega_0.
\end{equation}
On $\Omega_0$, the function $w:=\tilde u-u$ satisfies $\Delta w=q(x) w$ with $q(x) \ge 0$.
Hence
\[
w\le\max_{\partial \Omega_0} w \qquad\text{in }\;\Omega_0.
\]
Together with \eqref{eq5.11}, this shows that $w$ attains its global maximum at a point of  $\partial \Omega_0$.
Since $\partial \Omega_0 \subset\Omega$, the strong maximum principle gives
\[
\tilde u-u\equiv c
\]
for some constant $c\ge0$.
As in the proof of Theorem~\ref{thm3.9}, we obtain $c=0$.
Hence $u=\tilde u$, proving the theorem.
\end{proof}

\begin{remark}
The same argument applies to a bounded $C^3$ mean-convex domain.
Indeed, if $d(x)=\dist(x,\partial\Omega)$, then for all sufficiently small $s>0$, the domains $\Omega^s=\{d>s\}$ are smooth and mean-convex.
Thus $\Omega_0$ and the exhaustion $\{\Omega_j\}$ may be chosen among these domains.
For $\rho=\dist(\,\cdot\,,\Omega_0)$, we then have $\abs{\nabla\rho}=1$ and $\Delta\rho\ge0$ in $\Omega\setminus\overline\Omega_0$, so the arguments leading to \eqref{eq5.7} and \eqref{eq5.8} remain unchanged.
This recovers the result of \cite{CDG2012}.

The Keller--Osserman condition and convexity of $\sqrt F$ imply that $f$ is superadditive up to a constant; see Lemma~\ref{prop2.17}.
However, \cite{LGM2017} requires $C^1$ boundary and a uniform interior sphere condition, whereas the present result applies to arbitrary bounded convex domains.
\end{remark}

\section{Appendix}

We compare our structural assumptions with related criteria in the literature.

\begin{lemma}\label{prop2.10}
Assume that \eqref{eq3.1} holds for some $0<\alpha\le1$ and $t_0\in\mathbb R$.
Define
\[
g(\ell):=\inf_{s\ge0}\,\{f(s+\ell)-f(s)\},\qquad G(t):=\int_0^t g(\ell)\,d\ell.
\]
Then
\begin{equation}\label{eq2.11}
\int^\infty \frac{dt}{\sqrt{G(t)}}<\infty.
\end{equation}
\end{lemma}

\begin{proof}
We first consider $0<\alpha<1$.
We may assume that $F>0$ and $F^{-\alpha/2}f$ is nondecreasing on $[t_0,\infty)$.
Since $f$ is locally Lipschitz, for almost every $t\ge t_0$,
\begin{equation}\label{eq2.14}
(F^{-\alpha/2}f)' \ge 0 \iff f'(t)\ge \frac{\alpha f(t)^2}{2F(t)}.
\end{equation}
Applying Lemma~\ref{lem_convex} to
\[
h:=F^{1-\alpha/2}, 
\]
after increasing $t_0$ if necessary, we obtain
\[
\left(1-\frac{\alpha}{2}\right)F(t)^{-\alpha/2}f(t) = h'(t) \ge \frac{h(t)-h(t_0)}{t-t_0} \ge\frac{h(t)}{2t} = \frac{F(t)^{1-\alpha/2}}{2t},
\]
for all $t \ge t_0$.
Hence
\begin{equation}\label{eq2.15}
f(t)\ge c_0\frac{F(t)}{t},\qquad c_0>0.
\end{equation}
Combining \eqref{eq2.14} and \eqref{eq2.15}, we obtain for almost every $t \ge t_0$,
\begin{equation}\label{eq2.16}
f'(t)\ge c_1\frac{F(t)}{t^2},\qquad c_1>0.
\end{equation}
Since $F^{-\alpha/2}$ is nonincreasing and integrable at infinity,
\[
\frac t2 F(t)^{-\alpha/2} \le \int_{t/2}^t F(s)^{-\alpha/2}\,ds \to 0\qquad \text{as }\; t \to \infty.
\]
Thus, by increasing $t_0$ if necessary, and using \eqref{eq2.16}, for almost all $t \ge t_0$, we obtain
\[
f'(t)\ge c_2 t^{2/\alpha-2},\qquad c_2>0.
\]
If $s\ge t_0$, then for $\ell>0$,
\[
f(s+\ell)-f(s) \ge c_2\int_s^{s+\ell}t^{2/\alpha-2}\,dt
=\frac{c_2}{2/\alpha-1} \left((s+\ell)^{2/\alpha-1}-s^{2/\alpha-1}\right)
\ge \frac{c_2}{2/\alpha-1}\ell^{2/\alpha-1},
\]
since $2/\alpha-1>1$. If $0\le s<t_0$, monotonicity of $f$ gives
\[
f(s+\ell)-f(s)\ge f(\ell)-f(t_0),
\]
and integrating the same derivative estimate shows that for all sufficiently large $\ell$,
\[
f(\ell)-f(t_0)\ge c_3\ell^{2/\alpha-1},\qquad c_3>0.
\]

Taking the infimum over $s\ge 0$, we obtain
\[
g(\ell)\ge c_4\,\ell^{2/\alpha-1},\qquad c_4>0,
\]
for all sufficiently large $\ell$.
Consequently,
\[
G(r)\ge c r^{2/\alpha}
\]
for all $r$ sufficiently large.
Since $1/\alpha>1$, \eqref{eq2.11} follows.

It remains to consider $\alpha=1$.
Let $h=\sqrt{F}$.
Since $h$ is convex on $[t_0,\infty)$,
\[
f'=2(h')^2+2h''h \ge 2(h')^2
\] 
for almost every $t \ge t_0$.
If $s\ge t_0$, then, for all sufficiently large $\ell$,
\[
f(s+\ell)-f(s) \ge 2\int_s^{s+\ell} (h'(t))^2\,dt \ge 2\int_{s+\ell/2}^{s+\ell}(h'(t))^2\,dt
\ge \ell(h'(\ell/2))^2.
\]
If $0\le s<t_0$, then monotonicity of $f$ gives
\begin{equation}	\label{eq1703wed}
f(s+\ell)-f(s)\ge f(\ell)-f(t_0).
\end{equation}
Since $h>0$ and is convex on $[t_0,\infty)$, for all sufficiently large $\ell $,
\[
h(\ell) > h(\ell)-h(\ell/2) \ge \frac{\ell}{2}\,h'(\ell/2), \qquad h'(\ell)\ge h'(\ell/2),
\]
and therefore, for all sufficiently large $\ell$,
\[
f(\ell) =2h(\ell)\,h'(\ell) \ge \ell(h'(\ell/2))^2.
\]
Since $h'(t) \ge c$ for some constant $c>0$ for all sufficiently large $t$, 
\[
f(\ell)\to \infty \qquad\text{as }\;\ell \to \infty.
\]
Hence, the fixed term $f(t_0)$ in \eqref{eq1703wed} can be absorbed.
Taking the infimum over $s\ge0$, we obtain
\[
g(\ell)\ge c\,\ell(h'(\ell/2))^2
\]
for all sufficiently large $\ell$.
It follows that, for large $r$,
\[
G(r) \ge c\int_{r/2}^r \ell (h'(\ell/2))^2\,d\ell \ge c r^2(h'(r/4))^2 \ge c h(r/4)^2,
\]
where we applied Lemma~\ref{lem_convex} to $h=\sqrt F$ in the last step.
Since $h(r/4)^2=F(r/4)$, the Keller--Osserman condition yields \eqref{eq2.11}.
\end{proof}

\begin{lemma}\label{lem_convex}
Let $h:[t_0,\infty)\to\mathbb R$ be differentiable and convex, and suppose
\[
h(t) \to\infty \qquad \text{as }\; t\to\infty.
\]
Then there exists $t_1>t_0$ such that for all $t > t_1$,
\[
h'(t)\ge \frac{h(t)-h(t_1)}{t-t_1} \ge \frac{h(t)}{2t}.
\]
\end{lemma}
\begin{proof}
Fix $s>\max\{t_0,0\}$.
Since $h(t)\to\infty$, we may choose $t_1>s$ such that
\[
h(t_1)>0, \qquad h(s)\le \frac12\, h(t_1).
\]
By convexity,
\[
h'(t_1) \ge \frac{h(t_1)-h(s)}{t_1-s} \ge \frac{h(t_1)}{2t_1}.
\]
Hence, for $t>t_1$,
\[
h(t) \ge h(t_1)+(t-t_1)h'(t_1) \ge \frac{t+t_1}{2t_1}\,h(t_1).
\]
In particular,
\[
\frac{h(t_1)}{h(t)} \le \frac{2t_1}{t+t_1},
\]
and therefore
\[
\frac{h(t)-h(t_1)}{t-t_1}= \frac{h(t)}{t-t_1} \left(1-\frac{h(t_1)}{h(t)}\right) \ge
\frac{h(t)}{t+t_1} \ge \frac{h(t)}{2t}.
\]
Finally, convexity gives
\[
h'(t)\ge \frac{h(t)-h(t_1)}{t-t_1},
\]
which proves the result.
\end{proof}

\begin{lemma}\label{prop2.17}
Suppose that $\sqrt F$ is convex on $[t_0,\infty)$ for some $t_0\in\mathbb R$.
Then there exists $t_1>0$ such that $t\mapsto f(t)/t$ is nondecreasing on $[t_1,\infty)$.
Consequently, there exists $C\ge 0$ such that
\begin{equation}\label{eq2.18}
f(a+b)\ge f(a)+f(b)-C, \qquad a,\,b\ge0.
\end{equation}
\end{lemma}

\begin{proof}
Let $h:=\sqrt{F}$.
The monotonicity of $h$ gives
\[
\frac{t}{2h(t)}
\le \int_{t/2}^t\frac{ds}{h(s)},
\]
and hence $h(t)/t\to\infty$ by the Keller--Osserman condition for $f$.
We may assume $t_0>0$.
For $t>t_0$, convexity yields
\[
h'(t)\ge \frac{h(t)-h(t_0)}{t-t_0},
\]
and hence
\[
th'(t)-h(t) \ge \frac{t_0 h(t)-th(t_0)}{t-t_0}=\frac{tt_0}{t-t_0}\left( \frac{h(t)}{t}-\frac{h(t_0)}{t_0}\right)>0
\]
for all sufficiently large $t$, since $h(t)/t \to \infty$.
Thus $h(t)/t$ is nondecreasing for all sufficiently large $t$.
Moreover, $h'$ is nondecreasing by convexity and is positive for all sufficiently large $t$.
Therefore
\[
\frac{f(t)}t=2\,\frac{h(t)}{t}\, h'(t)
\]
is nondecreasing for all sufficiently large $t$, say for all $t\ge t_1$.
If $a, \,b\ge t_1$, then
\[
f(a)\le \frac{a}{a+b} \,f(a+b), \qquad f(b)\le \frac{b}{a+b}\,f(a+b),
\]
and therefore $f(a)+f(b)\le f(a+b)$.
If $0\le a<t_1$, monotonicity of $f$ gives
\[
f(a)+f(b)-f(a+b)\le f(a)\le f(t_1).
\]
By symmetry, this is true when $0\le b<t_1$.
Hence \eqref{eq2.18} holds with $C=f(t_1)$.
\end{proof}

Condition~\eqref{eq1.6} implies the profile condition used by Bandle and Ess\'en \cite{BE}.

\begin{lemma}	\label{lem2.24}
Assume
\[
\limsup_{t\to\infty} \psi(t)\frac{f(t)}{\sqrt{F(t)}}<\infty.
\]
Then, for every $0<\beta<1$,
\[
\liminf_{t\to\infty}\frac{\psi(\beta t)}{\psi(t)}>1.
\]
\end{lemma}

\begin{proof}
Since $f$ is nondecreasing,
\[
F(t)\le F(0)+t f(t),
\]
and hence $F(t)\le 2t f(t)$ for all sufficiently large $t$. Thus
\[
\frac{\psi(t)\sqrt{F(t)}}{t} = \psi(t)\frac{f(t)}{\sqrt{F(t)}}\frac{F(t)}{t f(t)} \le C
\]
for all sufficiently large $t$.
Consequently,
\[
\frac{\psi(\beta t)}{\psi(t)}-1 = \frac1{\psi(t)}\int_{\beta t}^{t}\frac{ds}{\sqrt{2F(s)}}
\ge \frac{(1-\beta)t}{\sqrt{2F(t)}\,\psi(t)} \ge \frac{1-\beta}{C\sqrt2}>0
\]
for all sufficiently large $t$.
\end{proof}

The next example shows that even together, \eqref{eq1.5} and \eqref{eq1.6} do not imply superadditivity up to a constant.
\begin{lemma}	\label{lem6.10}
For every $0<\alpha<1$, there exists $F\in C^\infty(\mathbb R)$, $F>0$, such that $f:=F'$ vanishes on $(-\infty,0]$, is strictly increasing on $(0,\infty)$, and satisfies
\begin{equation}\label{eq2.19}
\int^\infty F(t)^{-\alpha/2}\,dt<\infty, \qquad
F(t)^{-\alpha/2}f(t)\quad\text{nondecreasing},
\end{equation}
and
\begin{equation}\label{eq0732thu}
\limsup_{t\to\infty}\psi(t)\, \frac{f(t)}{\sqrt{F(t)}}<\infty, \qquad \psi(t):=\int_t^\infty \frac{ds}{\sqrt{2F(s)}},
\end{equation}
but
\begin{equation}\label{eq2.20}
\inf_{a,b\ge0} \left\{f(a+b)-f(a)-f(b)\right\}=-\infty.
\end{equation}
\end{lemma}

\begin{proof}
Fix $0<\alpha<1$. Choose $M>2$ so large that
\begin{equation}\label{eq2.21}
2\left(1-\frac2M\right)^{\alpha/(2-\alpha)} - \left(2-\frac2M\right)^{\alpha/(2-\alpha)}>0;
\end{equation}
this is possible since the left-hand side tends to $2-2^{\alpha/(2-\alpha)}>0$.
Choose $R>2M$ and set
\[
a_n:=R^n, \qquad b_n:=a_n^{(2-\alpha)/\alpha}, \qquad n\ge0.
\]
Let $\theta\in C^\infty(\mathbb R)$ be nondecreasing, with $\theta=0$ on $(-\infty,1/4]$ and $\theta=1$ on $[3/4,\infty)$.
Define a nondecreasing $\kappa\in C^\infty(\mathbb R)$ as follows.
Set $\kappa=0$ on $(-\infty,0)$, and choose $\kappa$ on $[0,2]$ so that
\[
\kappa(0)=0,\qquad \kappa(t)>0\quad(t>0),\qquad \kappa=1\quad\text{near }\;2,
\]
with all derivatives vanishing at $0$.
For $n\ge1$, set
\[
\kappa(t) = b_{n-1}+(b_n-b_{n-1})\, \theta\left(\frac{t-a_n}{a_n}\right),\qquad a_n\le t\le2a_n,
\]
and
\[
\kappa(t)=b_n, \qquad 2a_n\le t\le a_{n+1},\quad n\ge0.
\]
These definitions agree smoothly at the endpoints.
Set
\[
h(t):=1+\int_0^t\kappa(s)\,ds, \qquad t\in\mathbb R.
\]
Then $h\in C^\infty(\mathbb R)$, $h>0$, and $h'=\kappa$ is nonnegative and nondecreasing. Define
\[
F:=h^{2/(2-\alpha)}.
\]
Then
\[
f=F'= \frac{2}{2-\alpha}\, h^{\alpha/(2-\alpha)} h'.
\]
Thus $f=0$ on $(-\infty,0]$, while for $t>0$,
\[
f'(t) = \frac{2}{2-\alpha}\,h^{\alpha/(2-\alpha)}
\left(\frac{\alpha}{2-\alpha} \frac{(h')^2}{h} + h'' \right)>0.
\]
Moreover,
\[
F^{-\alpha/2}f= \frac{2}{2-\alpha}h',
\]
which is nondecreasing.
For $n\ge1$,
\[
h(a_n) \ge \int_{2a_{n-1}}^{a_n} \kappa(t)\,dt \ge 
b_{n-1}(a_n-2a_{n-1})=(R-2)R^{-2/\alpha}a_n^{2/\alpha}.
\]
Hence
\[
\int_{a_n}^{a_{n+1}} h(t)^{-\alpha/(2-\alpha)}\,dt \le (a_{n+1}-a_n)
\left((R-2)R^{-2/\alpha}a_n^{2/\alpha}\right)^{-\alpha/(2-\alpha)}= C a_n^{-\alpha/(2-\alpha)}.
\]
Since $a_n=R^n$, the resulting geometric series converges.
As
\[
F^{-\alpha/2}=h^{-\alpha/(2-\alpha)},
\]
this proves the first condition in \eqref{eq2.19}.

Set $t_n:=Ma_n$. Since $R>2M$,
\[
t_n,\,2t_n\in[2a_n,a_{n+1}],
\]
where $h'=b_n$. Thus
\[
h(2t_n)=h(t_n)+b_n t_n.
\]
Writing
\[
x_n:=\frac{h(t_n)}{b_n t_n},
\]
we obtain
\[
h(t_n)=b_n t_n x_n,\qquad h(2t_n)=b_n t_n (x_n+1).
\]
Also,
\[
h(t_n)=h(2a_n)+ \int_{2a_n}^{t_n} h'(t)\,dt \ge  b_n(t_n-2a_n),
\]
so
\[
x_n\ge \frac{t_n-2a_n}{t_n} =1-\frac{2a_n}{Ma_n}= 1-\frac{2}{M}>0.
\]
Since $0<\alpha/(2-\alpha)<1$, the function
\[
x \mapsto 2x^{\alpha/(2-\alpha)} -(x+1)^{\alpha/(2-\alpha)}
\]
is strictly increasing on $(0,\infty)$.
Therefore,
\[
\begin{aligned}
2f(t_n)-f(2t_n) &= \frac{2}{2-\alpha}\, b_n 
\left\{ 2h(t_n)^{\alpha/(2-\alpha)} - h(2t_n)^{\alpha/(2-\alpha)} \right\} \\
&= \frac{2}{2-\alpha}\, b_n^{2/(2-\alpha)} t_n^{\alpha/(2-\alpha)}
\left\{ 2x_n^{\alpha/(2-\alpha)} -(x_n+1)^{\alpha/(2-\alpha)} \right\} \\
&\ge \frac{2}{2-\alpha}\, b_n^{2/(2-\alpha)} t_n^{\alpha/(2-\alpha)}
\left\{ 2\left(1-\frac2M\right)^{\alpha/(2-\alpha)} - \left(2-\frac2M\right)^{\alpha/(2-\alpha)}
\right\}\\
&= \frac{2}{2-\alpha}\, M^{\alpha/(2-\alpha)}
\left\{ 2\left(1-\frac2M\right)^{\alpha/(2-\alpha)} - \left(2-\frac2M\right)^{\alpha/(2-\alpha)}
\right\} R^{\left(\frac{2}{\alpha}+\frac{\alpha}{2-\alpha}\right)n}.
\end{aligned}
\]
The bracket is positive by \eqref{eq2.21}, while the last factor tends to infinity.
Hence
\[
f(2t_n)-2f(t_n)\to-\infty,
\]
which proves \eqref{eq2.20}.

Finally, for $t\in[a_n,a_{n+1}]$, we have
\[
h(t)\ge h(a_n)\ge c a_n^{2/\alpha}, \qquad h'(t)\le b_n=a_n^{(2-\alpha)/\alpha}.
\]
Thus
\begin{equation}	\label{eq0654thu}
\frac{f(t)}{\sqrt{F(t)}} \le C a_n^{\frac{2-\alpha}{\alpha} -\frac{2(1-\alpha)}{\alpha(2-\alpha)}}
= C a_n^{\frac{\alpha^2-2\alpha+2}{\alpha(2-\alpha)}}.
\end{equation}
Moreover, for $m\ge n$ and $s\in[a_m,a_{m+1}]$,
\[
h(s)\ge c a_m^{2/\alpha},
\]
and hence
\[
\int_{a_m}^{a_{m+1}}\frac{ds}{\sqrt{F(s)}}
= \int_{a_m}^{a_{m+1}} h(s)^{-\frac1{2-\alpha}}\,ds
\le C a_m^{1-\frac{2}{\alpha(2-\alpha)}}
=C a_m^{-\frac{\alpha^2-2\alpha+2}{\alpha(2-\alpha)}}.
\]
Since $a_m=R^m$, summing the geometric series gives, for $t\in[a_n,a_{n+1}]$,
\begin{equation}	\label{eq0656thu}
\psi(t) \le
C a_n^{-\frac{\alpha^2-2\alpha+2}{\alpha(2-\alpha)}}.
\end{equation}
Combining \eqref{eq0654thu} and \eqref{eq0656thu} yields
\[
\psi(t)\,\frac{f(t)}{\sqrt{F(t)}}\le C
\]
for all sufficiently large $t$, proving \eqref{eq0732thu}.
\end{proof}


\end{document}